\documentclass[1p]{elsarticle}
\usepackage{lineno,hyperref}
\usepackage{amsmath}
\usepackage{amsfonts}
\usepackage{amsthm}
\usepackage{graphicx}
\usepackage{array,multirow,makecell}
\usepackage{subfigure}
\usepackage{appendix}
\modulolinenumbers[5]
\journal{arXiv }
\usepackage[font=small,labelfont=bf,labelsep=space]{caption}
\def\NN{\mathcal N}
\def\N{\mathbb N}  
\def\R{\mathbb R} 
\def\T{\mathcal{T}}
\def\D{\mathcal{D}}
\def\E{\mathcal{E}}
\def\d{\text{\rm d}}
\def\m{\text{\rm m}}
\newtheorem{lemma}{Lemma}[section]
\newtheorem{theo}{Theorem}[section]

\newtheorem{Def}{Definition}[section]
\newtheorem{prop}{Proposition}[section]
\newcommand{\figref}[1]{Fig.~\ref{#1}}
\begin{document}

\begin{frontmatter}

\title{A corrected decoupled  scheme for chemotaxis models}


\author{M. Akhmouch\corref{}}

\author{M. Benzakour Amine 
\corref{mycorrespondingauthor}}
\ead{benzakouramine.m@ucd.ac.ma
}
\cortext[mycorrespondingauthor]{Corresponding author}


\address{ D\'epartement de Math\'ematiques, Facult\'e des sciences et techniques de F\`es, USMBA, B.P. 2202 F\`es, Maroc }

\numberwithin{equation}{section}
\begin{abstract}
The main purpose of this paper is to present a new corrected decoupled  scheme combined with a spatial finite volume method for chemotaxis models. First, we derive the scheme for a parabolic-elliptic chemotaxis model arising in embryology.  We then establish the existence and uniqueness of the numerical solution, and we prove that it converges to a corresponding weak solution for the studied model.  In the last section, several numerical tests are presented by applying our approach to a number of chemotaxis systems. The obtained numerical results demonstrate the  efficiency of the proposed scheme and its effectiveness to capture different forms of spatial patterns.

\end{abstract}
\end{frontmatter}

\section{Introduction}
\label{s1}
Chemotaxis refers to a phenomenon that enables cells (or organisms) to migrate in response to a chemical signal. This process has sparked the interest of many scientists since it is encountered in several medical and   biological applications, such as bacteria aggregation, tumour growth, integumental patterns in animals  etc.

     In \cite{Oster1}, Oster and Murray discussed a cell-chemotaxis model involving motile cells that respond to a chemoattractant secreted by the cells themselves. In its dimensionless form, the model reads   

\begin{equation}
\left\{
\begin{aligned}
 &\partial_t u = \mu \Delta u - a\nabla \cdot (  \, u \nabla c),
 \\
&\partial_t c = \Delta c +\frac{u}{u+1} - c \,,
 \end{aligned}
 \right.
 \label{(1)}
\end{equation}
where $\mu$ and $a$ are positive constants, $u$ is the cell density and $c$ is the concentration of chemoattractant.

 The above system is based on the Keller-Segel  model \cite{Keller1}, which is the most popular model for chemotaxis. The migration of cells is assumed to be governed by Fickian diffusion and chemotaxis, and the mass of cells is conserved. The chemoattractant is assumed also to diffuse, but it increases with cell density in Michaelis-Menten way and undergoes decay through simple degradation.

In \cite{Murray1}, Murray \textit{et al.}  suggest that the presented  cell-chemotaxis model is an appropriate mechanism for the formation of stripe patterns on the dorsal integument of embryonic and hatchling alligators (\textit{Alligator mississippiensis}). These skin pigment patterns is associated with the density of melanocyte cells: Melanocytes are abundant in the regions where the black stripes appear, and are insufficient in the regions of the white stripes. The formation of these stripes is a result of a chemoattractant secretion. The system  \eqref{(1)} is known to produce propagating pattern of standings peaks and troughs in cell density in the case of one-dimensional space. This patterning process  was numerically and analytically investigated by Myerscough and Murray in \cite{Myerscough1}.

In this paper, we will first focus on the following parabolic-elliptic system

\begin{equation}
\left\{
\begin{aligned}
 &\partial_t u = \mu \Delta u - a\nabla \cdot (  \, u \nabla c)\quad \text{in}\ \Omega_T ,
 \\
&0= \Delta c +\frac{u}{u+1} - c \quad \text{in}\ \Omega_T,
 \end{aligned}
 \right.
 \label{(2)}
\end{equation}
with the homogeneous Neumann boundary and initial conditions
\begin{equation}
 \nabla u \cdot \nu=\nabla c \cdot \nu=0 \quad
\text{on} \ \partial \Omega \times (0,T_f),\quad u(.,0)=u_0   \quad \text{in} \ \Omega ,
\label{(3)}
\end{equation}
where $\Omega_T:=\Omega \times (0,T_f)$,  $\Omega \subset\R^2$ is an open bounded  polygonal subset, $T_f>0$ is a fixed time and $\nu$ denotes the outward unit-normal on the boundary $\partial \Omega$. The system  \eqref{(2)} is a simpler version of the original model  \eqref{(1)} in which the second equation of the system is elliptic, using the reasonable assumption that the  chemoattractant diffuses  much faster than  cells.

  In this work, we develop a decoupled finite volume scheme which can be applied to a class of chemotaxis models.  For the  convection-diffusion term, the approximation used is quite similar to  the hybrid scheme of Spalding \cite{Spalding1}. Concerning the time discretization, which is the main aim of this paper, it is developed such that the scheme  only requires to solve decoupled  systems, which excludes  fully implicit discretizations. We require also that the scheme converges  without needing to fulfill any CFL condition,  which is not the case of  the fully explicit schemes.  In the literature, a number of decoupled methods for the Keller-Segel model and its variants have been proposed (see, e.g., \cite{Saito2,Saito1,Strehl1,Andreianov1,Ibrahim1,Chamoun1,Akhmouch1,Akhmouch2}). In all these works, the time discretization is based on the classical backward Euler scheme with an explicit approximation of some terms to avoid coupling of the system. However, it is well known that the main drawback of this strategy is its lack of accuracy.  A more efficient approach will be presented in this work.
  
This paper provides also a convergence analysis of the proposed scheme applied to the system \eqref{(2)}--\eqref{(3)}. It is proved that the convergence of the approximate solution can be obtained  for any nonnegative initial cell density $u_0 \in L^2(\Omega)$. Our proof uses some techniques from \cite{Filbet1}, where a fully implicit upwind finite volume scheme is studied for the classical Keller-Segel model.

 The outline of this paper is as follows. In the next section, we present our corrected decoupled  finite volume scheme to approximate the solution of \eqref{(2)}--\eqref{(3)}.  In section 3, we prove the existence and uniqueness of the solution of the proposed scheme. Positivity preservation and mass conservation are also shown in this section. A priori estimates are given in Section 4. In Section 5, we use these estimates to prove  that the approximate solution converges to a weak solution of the studied model. In Section 6, we present some numerical tests  and compare the accuracy of our approach with that of more usual decoupled schemes. The paper ends with a conclusion.

\section{Presentation of the numerical scheme} 
\label{s2}
\subsection{Spatial discretization of $\Omega$, definitions and preliminaries}
  \label{s2.1}
We assume that $\Omega \subset\R^2$ is an open bounded  polygonal  subset. Following  Definition 9.1 in \cite{Eymard1}, we consider  an  admissible finite volume mesh of $\Omega$, denoted by $\T$.  This mesh is given by:

\begin{itemize}
\item[$\bullet$] A family  of control volumes which is commonly denoted by the same notation of the mesh $\T$. All control volumes are open and convex polygons.
\item[$\bullet$] A family $\E$ of edges , where the set of edges of any control volume $K\in \T$ is denoted by   $\E_K$. We denote also by $\sigma=K|L$ the  edge between $K$ and $L$ ($\sigma \in \E_K$ and $\sigma \notin \partial \Omega$).
 
\item[$\bullet$] A family of points $(x_K)_{K\in\T}$ such that $x_K \in \overline{K}$ (for all $K\in\T$).  The straight line going through $x_K$ and $x_L$ must be orthogonal to $\sigma=K|L$.
\end{itemize}

For all $K$ $\in \T$, we denote by $\NN (K)$ the set of control volumes which have a common  edge with $K$, and by m  the  Lebesgue measure in $\R^2$ or $\R$. 

 For all $\sigma \in \E_K$,   we define

$$
d_\sigma =\left\{\begin{array}{ll}
  \d(x_K,\sigma) &\quad\mbox{if }\sigma\subset   \partial \Omega, \\
  \d(x_K,x_L) &\quad\mbox{otherwise },\ \sigma=K|L,  
  \end{array}\right.
$$
where d is the Euclidean distance, and we denote by $\tau_\sigma$ the transmissibility coefficient given by:
$$
  \tau_\sigma = \frac{\m(\sigma)}{d_\sigma}, \quad \sigma\in\E.
$$

 The time discretization of $(0,T_f)$ is given by a uniform partition: $0=t_0<t_1<...<t_N=T_f$   with $N \in \N$ and   $t_n =n\Delta t$ for $n=0,...,N$.
 
 We denote by $h$ the maximal size (diameter) of the control volumes included in $\T$, and we define
$$
\delta=\text{max}(\Delta t, h).
$$

 In Section 4, the following time-step condition will be used: there exists $\alpha>0$ such that
 
 \begin{equation}
 1- 2 a\Delta t \geq \alpha.
  \label{(555)}
\end{equation}

 We will also need   this  additional constraint on the mesh: there exists $\xi>0$ such that
\begin{equation}
  \d(x_K,\sigma)\geq \xi \,\d(x_K,x_L),
  \label{(5)}
\end{equation}
which is specially needed to apply the discrete  Gagliardo-Nirenberg-Sobolev inequality (see Lemma 2.1). 

We define a weak solution of the system  \eqref{(2)} with boundary and initial conditions \eqref{(3)} as follows:
\begin{Def}
A weak solution of the  initial-boundary value problem  \eqref{(2)}--\eqref{(3)} is a pair of functions $(u,c)\in L^2(0,T;H^1(\Omega))^2 $ which verify the following identities for all test functions $\psi\in \D(\Omega\times[0,T_f))$ : 
\begin{gather}
  \int_0^T\int_\Omega\left(u\,\partial_t \psi - \mu \nabla u\cdot\nabla \psi + a\, u\nabla c\cdot\nabla \psi \right)\,dxdt
  + \int_\Omega u_0\,\psi(x,0)\,dx = 0,\label{(f1)}  \\
  \int_0^T\int_\Omega\nabla c\cdot\nabla\psi \ dxdt = \int_0^T\int_\Omega \left(\frac{u}{u+1}- c\right)\psi \,dxdt.\,\label{(f2)}
\end{gather}
\end{Def}

We denote by $X(\T)$  the set of functions from  $\Omega$ to $\R$ which are constant over each control volume of the mesh. Let  $1\leq p<\infty$, if $v \in X(\T)$, the corresponding discrete $L^p$ norm  reads
$$\|\,v\,\|_p = \left(\sum_{K\in\T}\m(K)\,|v_K|^p\right)^{1/p}, $$
where $v(x)=v_K$ for all $x \in K$ and for all $K \in \T$. 
 We define also the discrete $W^{1,p}$ seminorm and the discrete $W^{1,p}$ norm: 
 $$|\,v\,|_{1,p,\T} = \left(\sum_{\sigma\in\E}\frac{\m(\sigma)}{d_\sigma^{p-1}}\, 
  |D_\sigma v|^p\right)^{1/p}, 
$$

$$
\|\,v\,\|_{1,p,\T} = \|\,v\,\|_p + |\,v\,|_{1,p,\T} \, ,
$$ 
where for all $\sigma \in \E$, $D_\sigma v=0$ if $\sigma \subset \partial \Omega$ and $D_\sigma v=|v_K-v_L|$ otherwise, with $\sigma=K|L$.

We now recall the discrete Gagliardo-Nirenberg-Sobolev inequality (see \cite{chatard2}), which will be useful to establish a priori estimates in Section 4.

\begin{lemma} 
Let $\Omega$ be a an open bounded polyhedral domain of $\R^d$, $d\geq2$. Let $\T$ a mesh satisfying \eqref{(5)} and $v \in X(\T)$. 
\begin{itemize}
\item[$\bullet$] If $1\leq p < d$, let $1\leq s \leq r \leq p^*=pd/(d-p),$
\item[$\bullet$] If $p \geq d$, let $1\leq s \leq r <+\infty.$
\end{itemize}
Then there exists a constant $C>0$ only depending on $p,s,r,d$ and $\Omega$ such that 
$$\|\,v\,\|_r \, \leq \, \frac{C}{\xi^{(p-1)\theta /p}}\,\, \|\,v\,\|^{\theta}_{1,p,\T}\, \|\,v\,\|_s^{1-\theta},\ \forall v\in X(\T),$$
where $\theta$ is defined by 
$$
\theta=\dfrac{1/s-1/r}{1/s+1/d-1/p}.
$$
\end{lemma}

 \subsection{The corrected decoupled finite volume scheme }

  We begin this section by presenting a classical decoupled finite volume scheme   for the  problem      \eqref{(2)}--\eqref{(3)}:
\\
\\
for all $K\in \T$ and $n=0,...,N-1$

\begin{align}
& \m(K)\frac{u^{n+1}_K-u^n_K}{\Delta t}
  - \mu \sum_{\sigma\in\E_K}\tau_\sigma Du_{K,\sigma}^{n+1} \notag
  \\
  &+a \sum_{\substack{\sigma\in\E_{K}\\ \sigma=K|L}}\tau_\sigma\left(S\left( Dc_{K,\sigma}^{n+1}\right)u_{K}^{n+1}-S\left( -Dc_{K,\sigma}^{n+1}\right)u_{L}^{n+1}\right)=0, \label{(6)}
 \\
  & -\sum_{\sigma\in\E_K}\tau_\sigma\, Dc^{n+1}_{K,\sigma}
  +\m(K)\, c_K^{n+1}= \m(K)\, \frac{u_K^{n}}{u_K^{n}+1}, \label{(7)}
\end{align}
with the compatible initial condition
\begin{equation}
 u^0_K = \frac{1}{\m(K)}\int_K u_0(x)\,dx,\label{(8)}
 \end{equation}
 and where for all $\sigma \in \E_K$
$$ Dv_{K,\sigma}^n = \left\{\begin{array}{ll}
0 &\quad\mbox{for }\sigma\subset \partial \Omega,\\ 
  v_L^n-v_K^n &\quad\mbox{otherwise },\ \sigma=K|L.
  \end{array}\right. \\
  $$
  In the above scheme, the function $S$ is defined by
  \begin{equation}
 S(x) = \left\{\begin{array}{ll}
0, &\quad\mbox{if }x<2 \left(-\mu+\varepsilon \right)/a,\\ 
  \dfrac{x}{2}, &\quad\mbox{if }  |x| \leq 2 \left( \mu-\varepsilon \right)/a,\\
  x, &\quad\mbox{if },x>2 \left( \mu-\varepsilon \right)/a,
   \end{array}\right. \\ \label{(S)}
\end{equation}
where $\varepsilon$ is a small constant such that $\varepsilon\geq 0$ and $\varepsilon <<\mu$.

 The terms $u^n_K$ and $c^n_K$ denote respectively the approximations of the quantities  $\frac{1}{\m(K)}\int_K u(x,t^n)\,dx$ and $ \frac{1}{\m(K)}\int_K c(x,t^n)\,dx$.  As we can see, the proposed finite volume scheme is decoupled:  at each time-step, we begin by solving \eqref{(7)} to compute $c^{n+1}_K$ and then, we compute $u^{n+1}_K$ from \eqref{(6)}. The  discretization used for $\nabla \cdot ( u \nabla c) $ is equivalent to the second order central difference scheme when $ \left|Dc_{K,\sigma}^{n+1}\right| \leq  2 \left(\mu-\varepsilon\right)/a$  and to the first order upwind scheme when $Dc_{K,\sigma}^{n+1}< 2 \left(-\mu+\varepsilon\right)/a$ or $Dc_{K,\sigma}^{n+1}> 2 \left(\mu-\varepsilon\right)/a$.  When $\varepsilon=0$, the scheme is identical to that  of Spalding \cite{Spalding1} (see also \cite{Patankar1}).

 It is clear that we can obtain a best accuracy if we replace  \eqref{(7)} by the equation
 
 \begin{equation}
 -\sum_{\sigma\in\E_K}\tau_\sigma\, Dc^{n+1}_{K,\sigma}
  +\m(K)\, c_K^{n+1}= \m(K)\, \frac{u_K^{n+1}}{u_K^{n+1}+1},\label{(9)}
 \end{equation}
 however, it will be expensive in term of computational cost to find the solution of the scheme \eqref{(6)},\eqref{(9)} since we have to solve a large nonlinear system at each time step. 
 
Now, for all $K\in \T$ and $n=0,...,N-1$, we define  
 
\begin{equation}
T_K^{n+1}=\m(K)\left(\frac{u_K^{n+1}}{u_K^{n+1}+1}-\frac{u_K^{n}}{u_K^{n}+1}\right),\quad T_K^{0}=0 . \label{(11)}
 \end{equation} 
The equation \eqref{(9)} can  then be written as
 
\begin{equation}
 -\sum_{\sigma\in\E_K}\tau_\sigma\, Dc^{n+1}_{K,\sigma}
  +\m(K)\, c_K^{n+1}= \m(K)\, \frac{u_K^{n}}{u_K^{n}+1}+T_K^{n+1}.\label{(10)}
 \end{equation}

 As we can see, the only difference between   \eqref{(10)} and \eqref{(7)}  is the term $T_K^{n+1}$, so we conjecture that we can improve the accuracy of the decoupled scheme \eqref{(6)}--\eqref{(7)} if we add to the right hand side of \eqref{(7)} a correction term which approximates $T_K^{n+1}$. Hence, we propose to replace \eqref{(7)} with the following scheme:
 
\begin{equation}
 -\sum_{\sigma\in\E_K}\tau_\sigma\, Dc^{n+1}_{K,\sigma}
  +\m(K)\, c_K^{n+1}= \m(K)\, \frac{u_K^{n}}{u_K^{n}+1}+\beta_n T_K^{n},\label{(c)}
 \end{equation}  
where  $\beta_n>0$. The purpose of $\beta_n$ is to ensure the nonnegativity of the right-hand side of \eqref{(c)}, to do not affect the nonnegativity of $c^{n+1}_K$. Then, by supposing that $u_K^n\geq0$ (this will be proved in Section 3), we define $\beta_n$ for $n=0,...,N-1$  as follows:

\begin{equation}
\beta_n = \left\{\begin{array}{ll}
1, &\quad\mbox{if }\T^*_n = \emptyset,\\ 
  \underset{ {K \in \T^*_n}}{\text{min}} \cfrac{\cfrac{u_K^{n}}{u_K^{n}+1}}{\cfrac{u_K^{n-1}}{u_K^{n-1}+1}-\cfrac{u_K^{n}}{u_K^{n}+1}}, &\quad\mbox{otherwise },\\
   \end{array}\right. \\  \label{(beta)}
\end{equation}
 where $\T^*_n=\left\lbrace  K \in \T \, \middle| \, 2 \cfrac{u_K^{n}}{u_K^{n}+1}-\cfrac{u_K^{n-1}}{u_K^{n-1}+1}<0\right\rbrace  $ for $n=0,...,N-1$ and $\T^*_0 = \emptyset$. We can easily verify that $0<\beta_n\leq 1$. We mention however that, in practice, we will  take $\beta_n=1$, which seems the most natural choice. Indeed, several numerical tests are performed and, as expected, the  right-hand side  of \eqref{(c)} is always positive for this value  unless the time-step size is extremely large.  
 
  We define $u_{\delta},\: c_{\delta} $, the finite volume approximations of $u$ and $c$ by:
\begin{equation}
u_{\delta}(x,t)=u^{n+1}_K, \quad c_{\delta}(x,t)=c^{n+1}_K,  \quad x\in K,\,t\in [t^n,t^{n+1}).
\end{equation}

We define also approximations of the gradients of $u$ and $c$. To this end, we begin by defining $M_{K,\sigma}$, which is the cell formed from the vertices of $\sigma$, $x_K$ and $x_L$ if $\sigma=K|L \not \subset   \partial \Omega$, and from the vertices of $\sigma$ and $x_K$  if $\sigma \subset   \partial \Omega$.

Following \cite{Hillairet1}, we define the discrete gradient $dv_{\delta}$ which is the approximation of $\nabla v$  by
\begin{align*}
 &dv_{\delta}(x,t) = \frac{\m(\sigma)}{\m(M_{K,\sigma})}\,Dv_{K,\sigma}^{n+1}\,\,
  \nu_{K,\sigma}, \quad x\in M_{K,\sigma},\ t\in(t^n,t^{n+1}),\\
\end{align*}
for all $K\in \T$ and $n=0,...,N-1$, where $\nu_{K,\sigma}$ denotes the unit normal on $\sigma$ which is outward to $K$.

\section{Existence and uniqueness of a discrete solution}
\label{s3}
In this section, we prove existence and uniqueness of the  solution of the proposed scheme. We show also that the scheme is mass conserving and positivity preserving.

\begin{prop}
Assume  that $u_0 \geq 0$.
 Then there exists a unique solution  $\{(u_{K}^{n+1},c_{K}^{n+1}),\,K \in \T,\, n=0,...,N-1\} $ to the scheme \eqref{(6)},\eqref{(c)} which satisfies the following properties : 
 \begin{align}
 &u_{K}^{n+1}\geq 0 \quad \text{and} \quad  c_{K}^{n+1}\geq 0 \quad \text{for all}\ K \in \T,\,n=0,...,N-1 ,\label{(12)}
 \\
 &\sum_{K\in\T}\m(K)\,u_K^{n+1} = \sum_{K\in\T}\m(K)\,u_K^0 = \|u_0\|_{L^1(\Omega)}, \quad \text{for all } n=0,...,N-1.\label{(13)}
 \end{align}
\end{prop} 

\begin{proof}
For all $n=0,...,N-1$, we define the   vectors $U^{n+1}=\left( U^{n+1}_K \right)_{K \in \T}$, $F^{n}=\left( F^n_K \right)_{K \in \T}$, $C^{n+1}=\left( C^{n+1}_K \right)_{K \in \T}$ and $G^{n}=\left( G^n_K \right)_{K \in \T}$, for which:

 \begin{align*}
&U_{K}^{n+1}=u_{K}^{n+1},\quad F_{K}^{n}=\m(K)u_{K}^{n}/\Delta t,
\\
&C_{K}^{n+1}=c_{K}^{n+1} ,\quad  G_{K}^{n}=\m(K) \frac{u_K^{n}}{u_K^{n}+1}+\beta_n T_K^n.
\end{align*}
We also define   the matrix $A^n= \left(A_{K,K}^n \right)_{K \in \T}$ and $B= \left(B_{K,K} \right)_{K \in \T}$ by

 \begin{gather*}
 \begin{align*}
 &A_{K,K}^n=\m(K)/\Delta t + \sum_{\substack{\sigma\in\E_{K} \\ \sigma \not\subset \partial \Omega}}\tau_\sigma\left(\mu+a \, S\left( Dc_{K,\sigma}^{n+1}\right)\right) ,
 \\ &A_{K,L}^n=-\tau_\sigma\left(\mu+a \, S\left( -Dc_{K,\sigma}^{n+1}\right)\right), \quad \text{if}\, L \in \NN(K) \text{ with } \sigma=K|L,
\\  &A_{K,L}^n=0, \quad \text{otherwise},  \\
\text{and} 
 \\  &B_{K,K}= \sum_{\substack{\sigma\in\E_{K} \\ \sigma \not\subset \partial \Omega}} \tau_\sigma +\m(K),
\\  &B_{K,L}=-\tau_\sigma, \quad \text{if}\, L \in \NN(K) \text{ with } \sigma=K|L,
\\   &B_{K,L}=0, \quad \text{otherwise}.
 \end{align*}
 \end{gather*}

 The decoupled scheme \eqref{(6)},\eqref{(c)} can be then written equivalently as

\begin{align}  
 &A^{n}\,U^{n+1}=F^n,\label{(14)} \\
 &BC^{n+1}=G^n.\label{(15)}
\end{align}  
To prove the desired results, we proceed by induction on $n$. The proof in the case of $n=0$ is similar to that of the inductive step. We argue now that the vectors $U^{n}$ and $C^{n}$ are defined and nonnegative. We have

$$
|B_{K,K}| - \sum_{\substack{ L \in \T \\ L\neq K }}|B_{K,L}| =\m(K) > 0,     
 $$   
  which means the matrix $B$ is strictly diagonally dominant by rows. Moreover,  since the diagonal elements of $B$ are positive and its  off-diagonal  entries are nonpositive, we can conclude that $B$ is a nonsingular M-matrix, which implies the unique solvability of \eqref{(15)} with $B^{-1} >0$ . Therefore, in view of the nonnegativity of $G^n$ (by the definition of $\beta_n$ \eqref{(beta)}), we deduce that $C^{n+1} \geq 0$.
Since for all $\sigma=K|L \in \E_K$, $Dc_{L,\sigma}^{n+1}=-Dc_{K,\sigma}^{n+1}$, we have for all $K\in \T$ 
  
   $$
|A_{K,K}^n| - \sum_{\substack{ L \in \T \\ L\neq K }}|A_{L,K}^n| =\m(K)/\Delta t>0.     
 $$ 
 Then $A^n$ is strictly diagonal dominant by columns. Now, since $\mu+a \, S\left(x\right)\geq0$ for all $x \in \R$ (by the definition \eqref{(S)}), the matrix $A^n$ has nonpositive off-diagonal and positive diagonal entries, which implies that $A^n$ is a nonsingular M-matrix. Consequently, the existence and uniqueness of $U^{n+1}$ is proved, and since $F^{n} \geq 0$, it is clear that $U^{n+1} \geq 0$.

Now, summing \eqref{(6)} over $K \in \T$, we have:
$$ \sum_{K\in\T} \m(K) \left(u_{K}^{n+1}-u_K^n\right)=0.$$ 
Then, summing over $n=0,1,...,p$, with $p\leq N-1$ we get \eqref{(13)}:
$$\sum_{K\in\T}\m(K)\,u_K^{p+1} = \sum_{K\in\T}\m(K)\,u_K^0 = \|u_0\|_{L^1(\Omega)}, \quad \text{for all } p=0,...,N-1\, ,$$
which ends the proof.

  \end{proof}

\section{A priori estimates }
\label{s4}
 In all this section, we assume that $u_0 \in L^2(\Omega)$ and  $u_0 \geq 0$. We assume also that  $\{\left(u_{K}^{n+1},c_{K}^{n+1}\right),\,K \in \T,\, n=0,...,N-1 \} $  is the solution  of the scheme \eqref{(6)},\eqref{(c)}.
 \begin{prop}
 For all  $n=0,...,N-1$ and for all $K \in \T$, we have
\begin{align}
 & c_K^{n+1}  \leq 2,\label{(20)} \\
& \sum_{K\in\T}\sum_{\sigma\in\E_K}   \tau_\sigma  \left|Dc_{K,\sigma}^{n+1}\right|^2 \leq C.\label{(21)}
 \end{align}
where $C$ is a positive constant which only depends on $\Omega$. 
  \end{prop}

\begin{proof}

 Let $K$ be the control volume which verifies $c_K^{n+1}=\mbox{max} \lbrace c_L^{n+1} \rbrace_{L\in \T}$. Multiplying the equation \eqref{(c)} by  $\left(c_{K}^{n+1}-2 \right)^+$, and using the fact that $0<\beta_n\leq 1$, we get

\begin{align*}
-\sum_{\sigma\in\E_K} \tau_\sigma  Dc_{K,\sigma}^{n+1}
\left(c_{K}^{n+1}-2 \right)^+ &\leq \m(K)  \left(2-c_{K}^{n+1} \right) \left(c_{K}^{n+1}-2 \right)^+,
\end{align*}
it follows that
\begin{align*}
 -\sum_{\sigma\in\E_K} \tau_\sigma  Dc_{K,\sigma}^{n+1}
\left(c_{K}^{n+1}-2 \right)^+\leq0.
\end{align*}
 In view of the choice of $K$,  $Dc_{K,\sigma}^{n+1} \leq 0$ for all $ \sigma \in \E_K$, which leads to 
$$
-\sum_{\sigma\in\E_K} \tau_\sigma  Dc_{K,\sigma}^{n+1}
\left(c_{K}^{n+1}-2 \right)^+ \geq 0.
$$
 Hence, we deduce that $\left(c_{K}^{n+1}-2 \right) \leq 0$. This establish   \eqref{(20)}.

  Now, multiplying the equation \eqref{(c)} by $c_{K}^{n+1}$, summing over $K\in \T$,  using a summation by parts and \eqref{(20)}, we get
   
    $$
\dfrac{1}{2}\sum_{K\in\T}\sum_{\sigma\in\E_K} \tau_\sigma  \left|Dc_{K,\sigma}^{n+1}\right|^2 + \sum_{K\in\T}  \m(K) \left|c_K^{n+1}\right|^2 \leq 4 \m(\Omega),
   $$ 
which gives  \eqref{(21)}.
  
  \end{proof}

\begin{lemma}
 Assume that \eqref{(5)} is fulfilled and that $\varepsilon>0$ (see \eqref{(S)}). Then, there exists a constant $C>0$ only depending on $\Omega$, $T_f$,$u_0$, $a$, $\varepsilon$ and $\xi$  such that    
\begin{equation}
\sum_{n=0}^N\sum_{K\in\T}\Delta t \ \m(K) \left|u_K^{n}\right|^2 \leq C .\label{(17)}
\end{equation}
\end{lemma}

\begin{proof}
Performing the following changes of variable:
\begin{equation}
\tilde u_{K}^{n}=u_{K}^{n}+1\quad \text{for all}\ K \in \T,\,n=0,...,N ,\label{(22)}
\end{equation}
and
\begin{equation}
\tilde S(x)=S(x)+\left(\mu-\varepsilon \right)/a ,\label{(22b)}
\end{equation}
 and using the identity 
\begin{equation}
 S(x)-S(-x)=x,\label{(id)}
\end{equation}
the scheme \eqref{(6)} becomes
\begin{align}
& \m(K)\frac{\tilde u^{n+1}_K-\tilde u^n_K}{\Delta t}
  - \varepsilon \sum_{\sigma\in\E_K}\tau_\sigma D \tilde u_{K,\sigma}^{n+1} \notag
  \\
  &+a \sum_{\substack{\sigma\in\E_{K}\\ \sigma=K|L}}\tau_\sigma\left(\tilde S\left( Dc_{K,\sigma}^{n+1}\right) \tilde u_{K}^{n+1}-\tilde S\left( -Dc_{K,\sigma}^{n+1}\right)\tilde u_{L}^{n+1}\right)- a \sum_{\sigma\in\E_K}\tau_\sigma D c_{K,\sigma}^{n} = 0. \label{(23)}
\end{align}
From \eqref{(c)},  using \eqref{(20)} and since $0<\beta_n\leq 1$,  we can easily see that
\begin{equation*}
\sum_{\sigma\in\E_K}\tau_\sigma\, Dc^{n}_{K,\sigma}
  \leq 3\m(K).
\end{equation*}
Using the above inequality in \eqref{(23)}, we obtain
\begin{align*}
 &\m(K)\frac{\tilde u^{n+1}_K-\tilde u^n_K}{\Delta t}
  - \varepsilon \sum_{\sigma\in\E_K}\tau_\sigma D \tilde u_{K,\sigma}^{n+1} 
  \\
  &\leq -a \sum_{\substack{\sigma\in\E_{K}\\ \sigma=K|L}}\tau_\sigma\left(\tilde S\left( Dc_{K,\sigma}^{n+1}\right) \tilde u_{K}^{n+1}-\tilde S\left( -Dc_{K,\sigma}^{n+1}\right)\tilde u_{L}^{n+1}\right)+3 a\, \m(K)\, .
\end{align*}
 Let us now multiply the above inequality  by $\Delta t \, \log \left(\tilde u_K^{n+1} \right)$, and summing over $K\in \T$, we find that
 \begin{equation}
 E_1+E_2\leq E_3+E_4,
 \end{equation}
 where
\begin{align*}
&E_1=\sum_{K\in\T} \m(K)\left(\tilde u^{n+1}_K-\tilde u^n_K \right)\log \left(\tilde u_K^{n+1} \right),
\\
&E_2=-\varepsilon \sum_{K\in\T} \sum_{\sigma\in\E_K} \Delta t \, \tau_\sigma D \tilde u_{K,\sigma}^{n+1} \log \left(\tilde u_K^{n+1}\right),
\\
&E_3=-a \sum_{\substack{\sigma\in\E_{K}\\ \sigma=K|L}} \Delta t \, \tau_\sigma\left(\tilde S\left( Dc_{K,\sigma}^{n+1}\right) \tilde u_{K}^{n+1}-\tilde S\left( -Dc_{K,\sigma}^{n+1}\right)\tilde u_{L}^{n+1}\right) \log \left(\tilde u_K^{n+1}\right),
\\
&E_4=3a \sum_{K\in\T} \Delta t \, \m(K)  \log \left(\tilde u_K^{n+1}\right).
\end{align*}
From the expression of $E_1$, we can see that
\begin{align*}
E_1=&\sum_{K\in\T} \m(K)\left(\tilde u^{n+1}_K \log \left(\tilde u_K^{n+1}\right)-\tilde u^n_K \log \left(\tilde u_K^{n} \right) \right)
\\
&-\sum_{K\in\T} \m(K)\tilde u_K^{n} \left( \log \left(\tilde u_K^{n+1} \right) -
  \log \left(\tilde u_K^{n} \right)\right),
\end{align*}
and since we have from a Taylor expansion of $\log$ 
 $$\tilde u_K^{n} \left( \log \left(\tilde u_K^{n+1} \right) -
  \log \left(\tilde u_K^{n} \right)\right) \leq \tilde u_K^{n+1}-\tilde u_K^{n},$$
  we deduce using the mass conservation property \eqref{(13)} that  
  $$E_1 \geq \sum_{K\in\T} \m(K)\left(\tilde u^{n+1}_K \log \left(\tilde u_K^{n+1}\right)-\tilde u^n_K \log \left(\tilde u_K^{n} \right) \right).$$ 
  By an integration by parts on $E_2$, we get
  $$E_2=\dfrac{\varepsilon}{2} \sum_{K\in\T} \sum_{\sigma\in\E_K} \Delta t \, \tau_\sigma D \tilde u_{K,\sigma}^{n+1}  \left( \log \left(\tilde u_L^{n+1} \right) -
  \log \left(\tilde u_K^{n+1} \right)\right). $$
  Then, by a Taylor expansion of $\log$ we get 
  $$E_2=\dfrac{\varepsilon}{2} \sum_{K\in\T} \sum_{\sigma\in\E_K} \Delta t \, \tau_\sigma  \left| \dfrac{D \tilde u_{K,\sigma}^{n+1}}{\sqrt{\theta_\sigma^{n+1} }}\right|^2, $$
where $\theta_\sigma^{n+1}=t_\sigma \tilde u_{K}^{n+1} + \left(1-t_\sigma \right) \tilde u_{L}^{n+1} $ with $t_\sigma \in (0,1)$. Hence, using the fact that
$$\dfrac{\left |D \tilde u_{K,\sigma}^{n+1}\right|}{\sqrt{\theta_\sigma^{n+1} }}=\dfrac{\sqrt{\tilde u_{K}^{n+1}}+\sqrt{\tilde u_{L}^{n+1}}}{\sqrt{\theta_\sigma^{n+1} }}\left |D \left(\sqrt{\tilde u^{n+1}}\right)_{K,\sigma}\right| \geq \left |D \left(\sqrt{\tilde u^{n+1}}\right)_{K,\sigma}\right|,$$
we infer that
$$
E_2 \geq \dfrac{\varepsilon}{2} \sum_{K\in\T} \sum_{\sigma\in\E_K} \Delta t \, \tau_\sigma  \left |D \left(\sqrt{\tilde u^{n+1}}\right)_{K,\sigma}\right|^2.
$$
Using a summation by parts, we obtain
\begin{align*}
E_3^{n+1}=\dfrac{a}{2}\sum_{K\in\T} \sum_{\substack{\sigma\in\E_K \\ \sigma=K|L}} \Delta t\, &\tau_\sigma \left(\tilde S\left( Dc_{K,\sigma}^{n+1}\right)\tilde u_{K}^{n+1}-\tilde S\left( -Dc_{K,\sigma}^{n+1}\right)\tilde u_{L}^{n+1}\right)
\\
&\times \left(\log\left(\tilde u_{L}^{n+1}\right)-\log\left(\tilde u_{K}^{n+1}\right)\right).
\end{align*}

Now, we define $\tilde E_3^{n+1}$
\begin{align*}
\tilde E_3^{n+1}=\dfrac{a}{2}\sum_{K\in\T} \sum_{\substack{\sigma\in\E_K \\ \sigma=K|L}} \Delta t &\,\tau_\sigma  \theta_\sigma^{n+1} Dc_{K,\sigma}^{n+1}
\left(\log\left(\tilde u_{L}^{n+1}\right)-\log\left(\tilde u_{K}^{n+1}\right)\right).
\end{align*}
Next, using the identity \eqref{(id)} and the last expression of $E_3^{n+1}$, we can write
\begin{align*}
E_3^{n+1}-\tilde E_3^{n+1}=&\dfrac{a}{2}\sum_{K\in\T} \sum_{\substack{\sigma\in\E_K \\ \sigma=K|L}} \Delta t \, \tau_\sigma \tilde S\left( Dc_{K,\sigma}^{n+1} \right) \left(\tilde u_{K}^{n+1}-\theta_\sigma^{n+1}\right)
\\
&\times\left(\log\left(\tilde u_{L}^{n+1}\right)-\log\left(\tilde u_{K}^{n+1}\right)\right)
+\dfrac{a}{2}\sum_{K\in\T} \sum_{\substack{\sigma\in\E_K \\ \sigma=K|L}} \Delta t \, \tau_\sigma \tilde S\left(- Dc_{K,\sigma}^{n+1} \right)
\\
&\times \left(\theta_\sigma^{n+1}-\tilde u_{L}^{n+1}\right)
\left(\log\left(\tilde u_{L}^{n+1}\right)-\log\left(\tilde u_{K}^{n+1}\right)\right)
\\
=&\dfrac{a\left(1-t_\sigma \right)}{2} \sum_{K\in\T} \sum_{\substack{\sigma\in\E_K \\ \sigma=K|L}} \Delta t \, \tau_\sigma \tilde S\left( Dc_{K,\sigma}^{n+1} \right) \left(\tilde u_{K}^{n+1}-\tilde u_{L}^{n+1}\right)
\\
&\times\left(\log\left(\tilde u_{L}^{n+1}\right)-\log\left(\tilde u_{K}^{n+1}\right)\right)
+\dfrac{a\,t_\sigma}{2}\sum_{K\in\T} \sum_{\substack{\sigma\in\E_K \\ \sigma=K|L}} \Delta t \, \tau_\sigma \tilde S\left(- Dc_{K,\sigma}^{n+1} \right)
\\
&\times \left(\tilde u_{K}^{n+1}-\tilde u_{L}^{n+1}\right)
\left(\log\left(\tilde u_{L}^{n+1}\right)-\log\left(\tilde u_{K}^{n+1}\right)\right).
\end{align*}
Since we have from \eqref{(S)} that $S(x)\geq \left(-\mu+\varepsilon \right)/a$ $\forall x \in \R$, it follows that $\tilde S(x)\geq 0$. Consequently,  the above equality implies that $E_3^{n+1} \leq \tilde E_3^{n+1}$, which yields by a Taylor expansion of log to
$$
 E_3^{n+1} \leq \dfrac{a}{2}\sum_{K\in\T} \sum_{\sigma\in\E_K} \Delta t \,\tau_\sigma   Dc_{K,\sigma}^{n+1}
D\tilde u_{K,\sigma}^{n+1}.
$$
Multiplying \eqref{(c)} by $\tilde u_K^{n+1}$ and summing by parts we can easily see that
$$ 
E_3^{n+1} \leq 2 a \Delta t \,  \|u_0+1\|_{L^1(\Omega)}. 
$$
Finally for $E_4$, we have
$$
E_4\leq 3 a \, \Delta t \|u_0+1\|_{L^1(\Omega)}.
$$
Collecting the estimates obtained for $E_1$, $E_2$, $E_3$ and $E_4$, and summing over $n=0,...,N-1$, we obtain
\begin{align*}
\sum_{K\in\T}& \m(K)\tilde u^{N}_K \log \left(\tilde u_K^{N}\right)+ \dfrac{\varepsilon}{2} \sum_{n=0}^{N-1}\sum_{K\in\T} \sum_{\sigma\in\E_K} \Delta t \, \tau_\sigma  \left |D \left(\sqrt{\tilde u^{n+1}}\right)_{K,\sigma}\right|^2 
\\ &\leq
5a\,T \|u_0+1\|_{L^1(\Omega)}+\sum_{K\in\T} \m(K)\tilde u^{0}_K \log \left(\tilde u_K^{0}\right).
\end{align*}

Therefore, we have
 \begin{equation}
  \sum_{n=0}^{N-1}\sum_{K\in\T} \sum_{\sigma\in\E_K} \Delta t \, \tau_\sigma  \left |D \left(\sqrt{\tilde u^{n+1}}\right)_{K,\sigma}\right|^2 \leq C,\label{(25)}
   \end{equation}

where $C>0$ is a constant depending on $\Omega$, $T_f$,$u_0$, $a$ and $\varepsilon$ .

 Now, we denote  by $\tilde u_{\delta}^{n+1} \in X(\T)$ the function defined by: $
\tilde u_{\delta}^{n+1}(x)=\tilde u^{n+1}_K$ for all $ x\in K
$ and $n=0,...,N-1$,

From Lemma 2.1, we have
$$\left\|\,\sqrt{\tilde u_{\delta}^{n+1}}\,\right\|_4 \, \leq \, \frac{C}{\xi^{1/4}}\,\, \left\|\,\sqrt{\tilde u_{\delta}^{n+1}}\,\right\|^{1/2}_{1,2,\T}\, \left\|\,\sqrt{\tilde u_{\delta}^{n+1}}\,\right\|_2^{1/2},$$
which implies that
$$
\left\|\,\tilde u_{\delta}^{n+1}\,\right\|_2^2 \, \leq \, \frac{C}{\xi}\,\, \left\|\,\sqrt{\tilde u_{\delta}^{n+1}}\,\right\|^{2}_{1,2,\T}\, \left\|\,\tilde u_{\delta}^{n+1}\,\right\|_1,
$$
where $C>0$ depends on $\Omega$.

Finally, gathering the above inequality with \eqref{(25)} and \eqref{(13)}, we deduce that $u_{\delta}$ is bounded in $L^2\left(\Omega_T \right)$. This completes the proof of the lemma.
\end{proof}

 \begin{prop}
 Assume that \eqref{(5)} is fulfilled and that $\varepsilon>0$. Then, there exists  a constant $C>0$ depending on  $\Omega$, $T_f$,$u_0$, $a$, $\mu$ and $\xi$  such that, for all $n=0,1,...,p$, with $p\leq N-1$ : 
  \begin{align}
   &\sum_{K\in\T} \m(K) \left|u_K^{p+1}\right|^{2} + \sum_{n=0}^p\sum_{K\in\T}\sum_{\sigma\in\E_K}  \Delta t \ \tau_\sigma  \left|Du_{K,\sigma}^{n+1}\right|^2 \leq C,\label{(24)}
   \\
  & \sum_{n=0}^{p}\sum_{K\in\T} \m(K) \left(u_K^{n+1}-u_K^n \right)^2\leq C.\label{(cv)}
  \end{align}
  \end{prop}
  
  \begin{proof}

 Multiplying the equation \eqref{(6)} by $\Delta t \, u_K^{n+1}$, and summing over $K\in \T$, we have
 \begin{equation}
 I_1+I_2=I_3,\label{(4.5)}
 \end{equation}
 where
  \begin{gather*}
  \begin{align*}
 &I_1=\sum_{K\in\T} \m(K)u_K^{n+1} \left(u_{K}^{n+1}-u_K^n\right), 
 \\
 &I_2=-\mu\sum_{K\in\T}\sum_{\sigma\in\E_K}  \Delta t \ \tau_\sigma \ u_K^{n+1}Du_{K,\sigma}^{n+1},  
 \\
 &I_3=-a \sum_{K\in\T} \sum_{\substack{\sigma\in\E_{K}\\ \sigma=K|L}}\Delta t \ \tau_\sigma u_K^{n+1} \left(S\left( Dc_{K,\sigma}^{n+1}\right)u_{K}^{n+1}-S\left( -Dc_{K,\sigma}^{n+1}\right)u_{L}^{n+1}\right), 
  \end{align*}
  \end{gather*}
Performing a summation by parts on $I_2$. This yields,
$$I_2=\dfrac{\mu}{2}\sum_{K\in\T}\sum_{\sigma\in\E_K}  \Delta t \ \tau_\sigma \ \left|Du_{K,\sigma}^{n+1}\right|^2.$$ 
  A summation by parts on $I_3$ gives
  $$I_3=\dfrac{a}{2} \sum_{K\in\T} \sum_{\substack{\sigma\in\E_{K}\\ \sigma=K|L}}\Delta t \ \tau_\sigma  \left(S\left( Dc_{K,\sigma}^{n+1}\right)u_{K}^{n+1}-S\left( -Dc_{K,\sigma}^{n+1}\right)u_{L}^{n+1}\right)Du_{K,\sigma}^{n+1}.$$
Now, using   \eqref{(id)}, we can easily verify that
\begin{align*}
&2I_3-\dfrac{a}{2}\sum_{K\in\T} \sum_{\substack{\sigma\in\E_{K}\\ \sigma=K|L}} \Delta t \,\tau_\sigma 
  Dc^{n}_{K,\sigma} \left(u_L^{n+1}+u_K^{n+1}\right) Du_{K,\sigma}^{n+1}\\
  &\leq \dfrac{a}{2}\sum_{K\in\T} \sum_{\substack{\sigma\in\E_{K}\\ \sigma=K|L}}\tau_\sigma \Delta t \left(S\left( Dc_{K,\sigma}^{n+1}\right) +S\left( -Dc_{K,\sigma}^{n+1}\right) \right)\left(u_K^{n+1}-u_L^{n+1}\right)Du_{K,\sigma}^{n+1}\\
  &\leq 0,
\end{align*}
and since by multiplying \eqref{(c)} by $\left|u_{K}^{n+1}\right|^2$, and sum by parts, we have
 $$\dfrac{1}{2}\sum_{K\in\T}\sum_{\substack{\sigma\in\E_K \\ \sigma=K|L}} \ \tau_\sigma  Dc_{K,\sigma}^{n+1}\left(\left|u_L^{n+1}\right|^2-\left|u_K^{n+1}\right|^2\right)\leq 2 \sum_{K\in\T} \m(K) \left|u_K^{n+1}\right|^2,$$
we deduce that
  
  \begin{equation}
  I_3 \leq  a  \Delta t \sum_{K\in\T} \m(K)  \left|u_K^{n+1}\right|^2. \notag
  \end{equation}

  Employing the previous estimates for  $I_2$ and $I_3$ in \eqref{(4.5)}, and summing over $n=0,1,...,p$, with $p\leq N-1$ , we obtain

  \begin{align}
   \dfrac{1}{2}&\sum_{K\in\T} \m(K) \left|u_K^{p+1}\right|^{2} + \dfrac{1}{2} \sum_{n=0}^{p}\sum_{K\in\T} \m(K) \left(u_K^{n+1}-u_K^n \right)^2 \notag
   \\
   &+\dfrac{\mu}{2}\sum_{n=0}^p\sum_{K\in\T}\sum_{\sigma\in\E_K}  \Delta t \ \tau_\sigma  \left|Du_{K,\sigma}^{n+1}\right|^2 \notag
   \\
   &\leq \dfrac{1}{2}\sum_{K\in\T} \m(K) \left|u_K^0\right|^{2} +a \sum_{n=1}^{p+1}\sum_{K\in\T} \Delta t \, \m(K)\left|u_K^{n}\right|^2. \label{(4.6)} 
  \end{align}
  It follows from Lemma 4.1 that the right hand side of the above inequality is bounded, which gives the desired results. 
  \end{proof}

\begin{prop}
 Assume that the time step condition \eqref{(555)} is fulfilled. Then, the estimates \eqref{(24)} and \eqref{(cv)} hold for a constant $C>0$  depending on  $\Omega$, $T_f$, $u_0$, $a$ , $\mu$ and $\alpha$.
  \end{prop}

\begin{proof}
From \eqref{(4.6)}, and since  \eqref{(555)} is verified, we have
$$
\sum_{K\in\T} \m(K) \left|u_K^{p+1}\right|^{2} \leq
  \sum_{n=1}^{p}  \dfrac{2a \Delta t}{\alpha} \sum_{K\in\T}  \, \m(K)\left|u_K^{n}\right|^2+ \dfrac{\|u_0\|_2}{\alpha}  .
$$
By the discrete Gronwall inequality, we get

$$
\sum_{K\in\T} \m(K) \left|u_K^{p+1}\right|^{2} \leq
 \dfrac{\|u_0\|_2}{\alpha} \exp \left( \dfrac{2a T_f}{\alpha}    \right),
$$
which with \eqref{(4.6)} completes the proof.
\end{proof}

\section{Convergence of the finite volume scheme} 
\label{s5}

  \begin{prop}
   Assume that $u_0 \in L^2(\Omega)$ and  $u_0 \geq 0$. Assume that the time step condition \eqref{(555)} is fulfilled or that  \eqref{(5)} is fulfilled and  $\varepsilon>0$ (see \eqref{(S)}). Then, there exists a subsequence of $(u_{\delta},c_{\delta})_{\delta}$, not relabeled, and a function $(u,c)\in L^2(0,T;H^1(\Omega))^2$ such that as $\delta$ goes to zero we have
  \begin{gather*}
  u_{\delta} \to u \quad\text{strongly in }L^2(\Omega_T), \\
  du_{\delta} \rightharpoonup \nabla u \quad\text{weakly in }L^2(\Omega_T), \\
  c_{\delta}\rightharpoonup c, \quad  dc_{\delta}\rightharpoonup \nabla c
  \quad\text{weakly in }L^2(\Omega_T).
\end{gather*}  
   \end{prop}

\begin{proof}
The proof will be omitted since it is exactly similar to that of Proposition 4.1 in \cite{Filbet1}. We need only to note that $S(x)\leq |x|$ for all $x\in \R$.
\end{proof}

\begin{theo}
 Assume that $u_0 \in L^2(\Omega)$ and  $u_0 \geq 0$. Assume that the time step condition \eqref{(555)} is fulfilled or that  \eqref{(5)} is fulfilled and  $\varepsilon>0$ (see \eqref{(S)}). The function $(u,c)$ constructed in Proposition 5.1 is a weak  solution of \eqref{(2)}--\eqref{(3)} in the sense of Definition 2.1 .
\end{theo}

\begin{proof}
Let $\psi \in \D(\Omega\times[0,T_f))$, and let

\begin{align*}
&J_{10}(\delta)=-\int_0^T\int_\Omega u_\delta \,\frac{\partial \psi}{\partial t} dx \, dt-
 \int_\Omega u_\delta (x,0)\,\psi(x,0)\,dx, \\
&J_{20}(\delta)=\mu \int_0^T\int_\Omega d u_\delta \cdot\nabla \psi \ dx \, dt, \\
&J_{30}(\delta)=-a \ \int_0^T\int_\Omega u_\delta \ d c_\delta \cdot\nabla \psi \ dx \, dt, \\
\end{align*}

with $\lambda(\delta)=-\left(J_{10}(\delta)+J_{20}(\delta)+J_{30}(\delta)\right).$

Let $\psi_K^n=\psi\left(x_K,t^n\right)$ for all $K \in \T$ and $n=0,1,...,N$. Multiplying the scheme \eqref{(6)} by $\Delta t \psi_K^n$ and summing for $K$ and $n$, we have
$$J_{1}(\delta)+J_{2}(\delta)+J_{3}(\delta)=0$$
where
\begin{align*}
 &J_1(\delta)= \sum_{n=0}^{N-1}\sum_{K\in\T} \m(K) \left(u_{K}^{n+1}-u_K^n\right)\psi_K^n, 
 \\
 &J_2(\delta)=-\mu \sum_{n=0}^{N-1}\sum_{K\in\T}\sum_{\sigma\in\E_K}  \Delta t \ \tau_\sigma \ Du_{K,\sigma}^{n+1} \psi_K^n,   
 \\
 &J_3(\delta)= a \sum_{n=0}^{N-1} \sum_{K\in\T} \sum_{\substack{\sigma\in\E_{K}\\ \sigma=K|L}}\Delta t \ \tau_\sigma  \left(S\left( Dc_{K,\sigma}^{n+1}\right)u_{K}^{n+1}-S\left( -Dc_{K,\sigma}^{n+1}\right)u_{L}^{n+1}\right)\psi_K^n.
  \end{align*}
   From the weak convergence of $\left(dc_\delta \right)_\delta$ to $\nabla c$, $\left(du_\delta \right)_\delta$ to $\nabla u$, and  the strong convergence of $\left(u_\delta \right)_\delta$ to $u$ in $L^2(\Omega_T)$,  we have  when $\delta  \to 0$
  $$\lambda(\delta) \to \int_0^T\int_\Omega\left(u\,\frac{\partial \psi}{\partial t} - \mu \nabla u\cdot\nabla \psi + a\, u\nabla c\cdot\nabla \psi \right)\,dxdt
  + \int_\Omega u_0\,\psi(x,0)\,dx.$$
  To prove that $(u,c)$ defined in Proposition 5.1 verifies \eqref{(f1)}, we need to prove that $J_i(\delta)-J_{i0}(\delta) \to 0$ as $\delta \to 0$ for $i=1,2,3$. The proof is exactly similar to that of Proposition 4.3 in  \cite{Filbet1}, we simply have to note that, using the identity (for $\sigma=K|L$):
\begin{align*}
 S\left( Dc_{K,\sigma}^{n+1}\right)u_{K}^{n+1}-S\left( -Dc_{K,\sigma}^{n+1}\right)u_{L}^{n+1}=-S\left( -Dc_{K,\sigma}^{n+1}\right)Du_{K,\sigma}^{n+1}+u_K^{n+1}Dc_{K,\sigma}^{n+1}.
\end{align*}
we can write $J_3(\delta)=J_{31}(\delta)+J_{32}(\delta)$, with 
 
\begin{align*} 
 &J_{31}(\delta)=\dfrac{a}{2}\sum_{n=0}^{N-1}\sum_{K\in\T}\sum_{\sigma\in\E_K}  \Delta t \ \tau_\sigma S\left( -Dc_{K,\sigma}^{n+1}\right)Du_{K,\sigma}^{n+1} D \psi_{K,\sigma}^{n},\\
 & J_{32}(\delta)=-\dfrac{a}{2}\sum_{n=0}^{N-1}\sum_{K\in\T}\sum_{\sigma\in\E_K}  \Delta t \ \tau_\sigma u_K^{n+1}Dc_{K,\sigma}^{n+1} D \psi_{K,\sigma}^{n}.
\end{align*} 
The limit $\delta \to 0$ in the equation \eqref{(c)} is performed as in  the proof of \cite[Proposition 4.2]{Filbet1}. The treatment of the term $\dfrac{u_K^n}{u_K^n+1}$ is obvious, since it is bounded for all $K\in \T$ and $n=0,...,N-1$. For the correction term, using the Cauchy-Schwarz inequality and the fact that $0<\beta_n\leq 1$, we obtain 

\begin{align*}
\sum_{n=0}^{N-1}\sum_{K\in\T} \Delta t \beta_n T_K^n \psi_{K}^{n}\leq &\sqrt{\Delta t} \left( \sum_{n=0}^{N-1}\sum_{K\in\T} \Delta t \,  \m(K) \left|\psi_{K}^{n} \right|^2 \right)^{1/2}\\
& \times\left( \sum_{n=0}^{N-1}\sum_{K\in\T}   \m(K)\left| \frac{u_K^{n}}{u_K^{n}+1}-\frac{u_K^{n-1}}{u_K^{n-1}+1}\right|^2 \right)^{1/2}.
\end{align*}
Now,  since the function $x\longrightarrow \dfrac{x}{x+1}$ is smooth for all $x\geq0$, there exists a constant $C$ such that
\begin{equation*}
\left| \frac{u_K^{n}}{u_K^{n}+1}-\frac{u_K^{n-1}}{u_K^{n-1}+1}\right| \leq C \left| u_K^{n}-u_K^{n-1}\right|.
\end{equation*} 
From the two last inequalities, and from the regularity of $\psi$ and the estimate \eqref{(cv)}, it is easy to see that
\begin{equation*}
\sum_{n=0}^{N-1}\sum_{K\in\T} \Delta t \beta_n T_K^n \psi_{K}^{n}\  \to 0 \mbox{\quad as } \delta  \to 0.
\end{equation*}
This completes the proof of Theorem 5.1. 
\end{proof}

 \section{Numerical experiments}
In this section, we use the proposed corrected decoupled scheme to solve a number of two-dimensional chemotaxis systems. The obtained numerical results are compared with those of more usual decoupled methods. We mention that the computational time of all  compared schemes is almost the same, so the comparison is only performed  with respect to accuracy. For the value of $\varepsilon$,  reference solutions are computed with $\varepsilon=0$, and we take $\varepsilon=10^{-6}$ for the other computed solutions. Finally, as mentioned in Section 2.2, $\beta_n=1$ in all numerical tests.

\textit{Test 1.} This test deals with the comparison between the corrected decoupled scheme \eqref{(6)},\eqref{(c)} and the decoupled scheme \eqref{(6)}--\eqref{(7)}. The spatial  domain is $\Omega=(-7/2,7/2)\times(-35,35) $ and is discretized via  a uniform mesh of $12250$ control volumes, whereas the final time is $T_f=150$. We adopt the following parameters used in \cite{Myerscough1}: $\mu=0.25$ and $a=2$, and we consider the following initial condition:

\begin{equation}
  u_0(x) = \left\{\begin{array}{ll}
1+\epsilon(x) &\quad\mbox{if } x\in (-9/2,9/2)\times(-1,1),\\ 
 1 &\quad\mbox{otherwise, with  } x \in \Omega,
  \end{array}\right. \\
  \label{(init)}
  \end{equation}
  where $\epsilon(x)$ is a positive perturbation  which is constant on each control volume. It is equal on each cell of the mesh  to the average of ten uniformly distributed random values in $[0,1]$.
  
   Since the exact  cell density solution $u$ of the studied system \eqref{(2)}--\eqref{(3)} is unavailable, we compute a reference solution by the proposed corrected decoupled scheme  on a very fine  time-stepping $\Delta t=10^{-3}$, so the number of time-steps needed to reach the final time $T_f$ is $150,000$. We then use the obtained reference solution to compute the relative $L^2$-errors for the two schemes  \eqref{(6)},\eqref{(c)} and  \eqref{(6)}--\eqref{(7)}. These errors are presented in Table \ref{tab.1}. From this table, we can see that both schemes  are first-order accurate and stable for any  time step $\Delta t$ used in this test. However, the corrected  decoupled scheme is about three to five times more accurate than the decoupled scheme \eqref{(6)}--\eqref{(7)}.

\begin{table}[h!]
\begin{center}
\begin{tabular}{l c c c c} 
\hline
\\
    {$\Delta t$ } & $L^2$-error & Rate & $L^2$-error & Rate \\
                  & coorected decoupled  &       & decoupled &         \\
\hline  
\\  
     {$5$} &  $1.320 \times 10^{-1}$ & --- & $4.042 \times 10^{-1}$ & --- \\
     {$1$} & $2.923 \times 10^{-2}$ & $0.937$ & $1.435 \times 10^{-1}$ & $0.643$  \\ 
     {$5.10^{-1}$} & $1.703 \times 10^{-2}$ & $0.779$ & $7.767 \times 10^{-2}$ & $0.886$ \\ 
     {$10^{-1}$} & $3.817 \times 10^{-3}$ & $0.929$ & $1.630 \times 10^{-2}$ & $0.970$ \\
     {$5.10^{-2}$} & $1.918 \times 10^{-3}$ & $0.993$ & $8.205 \times 10^{-3} $ & $0.990$ \\
     {$10^{-2}$} &  $3.566 \times 10^{-4}$ & $1.05$ & $1.672 \times 10^{-3}$ & $0.988$ \\
\hline     
\end{tabular}
\caption{  Relative $L^2$-errors and time convergence orders obtained for $(u)$  using the corrected decoupled  scheme \eqref{(6)},\eqref{(c)} and  the decoupled scheme \eqref{(6)}--\eqref{(7)}.  }
\label{tab.1}
\end{center}
\end{table}

The initial condition \eqref{(init)} and  the reference solution at final time $T_f$ are plotted in \figref{Fig.1}. The presented figure demonstrates the ability of the model \eqref{(2)}--\eqref{(3)} to generate stripe patterns. Three-dimensional plots of the cell density using both studied schemes at final time with $\Delta t=10^{-2}$ are presented in  \figref{Fig.2}. Finally, in \figref{Fig.3}, we plot the contour along the line $L=\lbrace0\rbrace\times(-35,35)$ of the reference solution and of the solutions computed using both schemes with $\Delta t=1$. We can see from this figure that even for a large time-step  size, the corrected decoupled scheme is in close agreement with the reference solution, which is not the case of the scheme \eqref{(6)}--\eqref{(7)}.

\begin{figure}[h!]
\subfigure{\includegraphics[width=6.cm]{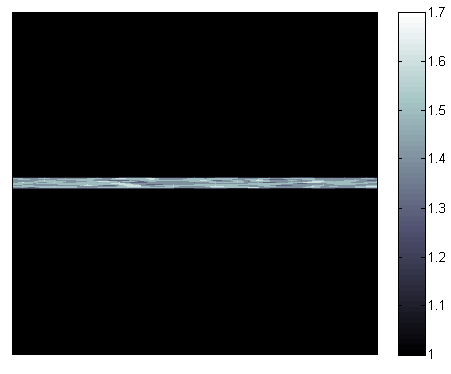}}
\subfigure{\includegraphics[width=6.cm]{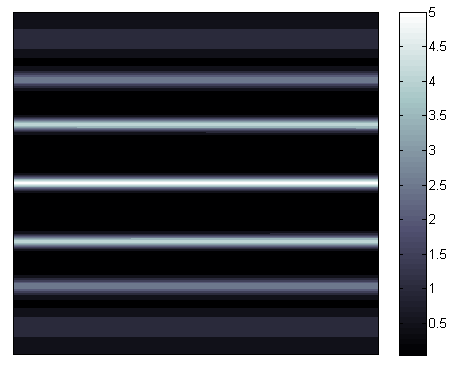}}
\caption{Initial cell density \eqref{(init)} (left)  and cell density $(u)$ at final time $T_f$  computed via the corrected decoupled scheme \eqref{(6)},\eqref{(c)}  with $\Delta t=10^{-2}$ (right).}
\label{Fig.1}
\end{figure}

\begin{figure}[h!]
\subfigure{\includegraphics[width=6.cm]{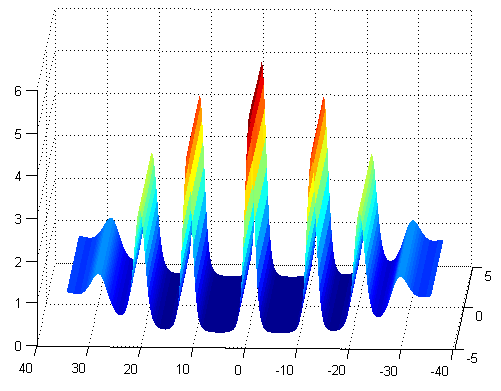}}
\subfigure{\includegraphics[width=6.cm]{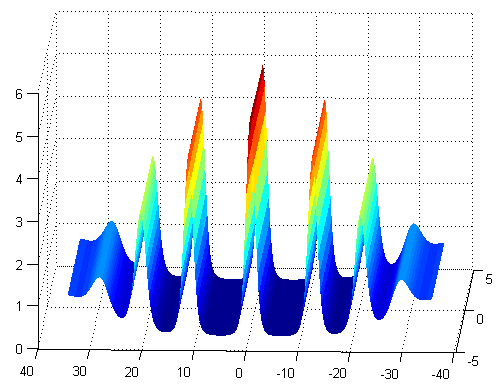}}
\caption{Three-dimensional plot of the cell density $(u)$ at final time $T_f$  computed via the corrected decoupled scheme \eqref{(6)},\eqref{(c)} (left) and the decoupled scheme \eqref{(6)}--\eqref{(7)} (right) with $\Delta t=10^{-2}$. }
\label{Fig.2}
\end{figure}

\begin{figure}[h!]
\begin{center}
\includegraphics[width=1\textwidth,height=0.4\textheight]{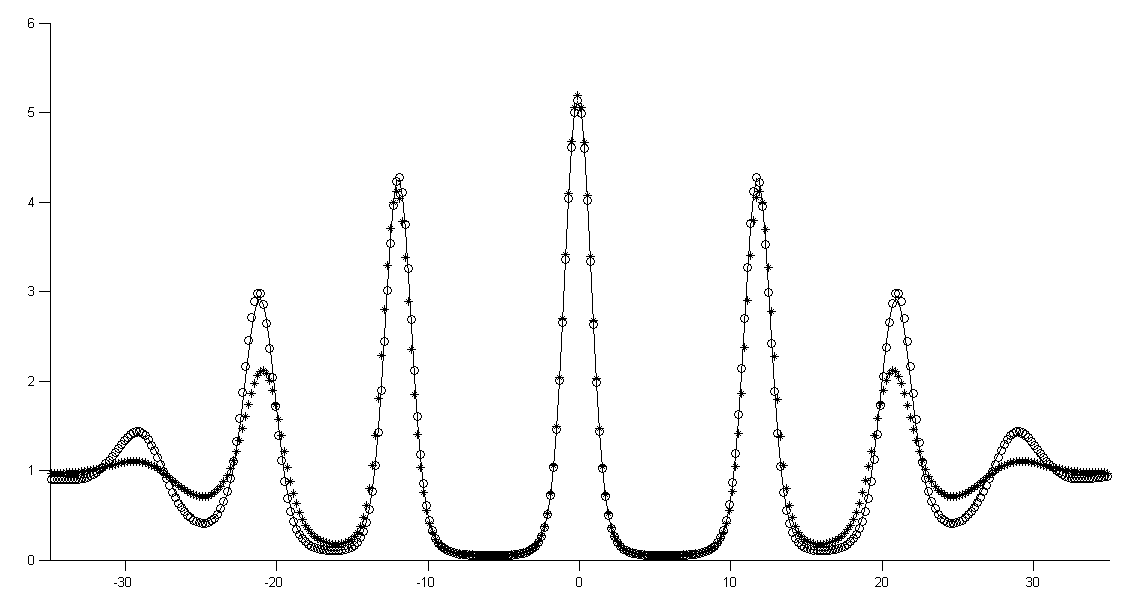}
\caption{Contours along the line $L$ of the solution $(u)$ at final time $T_f$; solid line reference solution, 
$\circ$ solution  computed via the corrected decoupled scheme \eqref{(6)},\eqref{(c)} with $\Delta t=1$, $*$  solution  computed via the decoupled scheme \eqref{(6)}--\eqref{(7)} with $\Delta t=1$.}
\label{Fig.3}
\end{center}
\end{figure}

\textit{Test 2.} The purpose of this test is to investigate the efficiency of the corrected decoupled scheme in the case of the parabolic-parabolic version of the model \eqref{(2)}--\eqref{(3)}.

We consider the same data used in the previous test. Nevertheless, we need to define an initial condition for the concentration of chemoattractant, we take then $c(x,0)=1/32$. The schemes adopted are similar to those of the previous test, we only need to replace the equation \eqref{(7)} in the scheme \eqref{(6)}--\eqref{(7)} by
\begin{equation}
 \m(K)\frac{c^{n+1}_K-c^n_K}{\Delta t}-\sum_{\sigma\in\E_K}\tau_\sigma\, Dc^{n+1}_{K,\sigma}
  +\m(K)\, c_K^{n+1}= \m(K)\, \frac{u_K^{n}}{u_K^{n}+1},
  \label{T1}
\end{equation}

and we substitute the equation \eqref{(c)} in the scheme \eqref{(6)},\eqref{(c)} by (we recall that $\beta_n=1$)

\begin{equation}
  \m(K)\frac{c^{n+1}_K-c^n_K}{\Delta t}-\sum_{\sigma\in\E_K}\tau_\sigma\, Dc^{n+1}_{K,\sigma}
  +\m(K)\, c_K^{n+1}= \m(K)\, \frac{u_K^{n}}{u_K^{n}+1}+ T_K^{n}.
 \label{T2} 
\end{equation}  
The following decoupled scheme is also investigated in this test:
\begin{align}
& \m(K)\frac{u^{n+1}_K-u^n_K}{\Delta t}
  - \mu \sum_{\sigma\in\E_K}\tau_\sigma Du_{K,\sigma}^{n+1} \notag
  \\
  &+a \sum_{\substack{\sigma\in\E_{K}\\ \sigma=K|L}}\tau_\sigma\left(S\left( Dc_{K,\sigma}^{n}\right)u_{K}^{n+1}-S\left( -Dc_{K,\sigma}^{n}\right)u_{L}^{n+1}\right)=0, \label{T3}
 \\
  &  \m(K)\frac{c^{n+1}_K-c^n_K}{\Delta t}-\sum_{\sigma\in\E_K}\tau_\sigma\, Dc^{n+1}_{K,\sigma}
  +\m(K)\, c_K^{n+1}= \m(K)\, \frac{u_K^{n+1}}{u_K^{n+1}+1}, \label{T4}
\end{align}
such scheme is used for example in \cite{Strehl1}.

The reference solution is computed similarly to the previous test. We observe from Table \ref{tab.2} that the scheme   \eqref{T3}--\eqref{T4} is the less accurate one. We can see also that  the corrected decoupled scheme    \eqref{(6)},\eqref{T2} is about four to five times more accurate than the scheme \eqref{(6)},\eqref{T1}.

\begin{table}[h!]
\begin{center}
\begin{tabular}{l c c c c c c} 
\hline
\\
    {$\Delta t$ } & $L^2$-error & Rate & $L^2$-error & Rate & $L^2$-error & Rate \\
                  & coorected decoupled  &       &  \eqref{(6)},\eqref{T1} &        &  \eqref{T3}--\eqref{T4} &         \\
\hline  
\\  
     {$5$} & $8.450\times 10^{-2}$ & --- & $3.775 \times 10^{-1}$ & --- &$4.231 \times 10^{-1}$ & --- \\
     {$1$} & $2.323\times 10^{-2}$ & $0.802$ & $1.117 \times 10^{-1}$ &  $0.756$ & $1.243 \times 10^{-1}$& $0.761$ \\ 
     {$5.10^{-1}$} & $1.344 \times 10^{-2}$ & $ 0.790$ & $6.109\times 10^{-2}$ & $0.871$ & $6.846 \times 10^{-2}$ &$0.861$ \\ 
     {$10^{-1}$} & $2.971\times10^{-3}$ & $0.938$ & $1.314\times 10^{-2}$ & $0.955$ & $1.477 \times 10^{-2}$&$0.953$ \\
     {$5.10^{-2}$} & $1.490\times 10^{-3}$ & $0.996$ & $6.634\times 10^{-3}$ & $0.986$ &$7.456\times 10^{-3}$ &$0.986$ \\
     {$10^{-2}$} & $2.765 \times 10^{-4}$ & $ 1.05$ & $1.354\times 10^{-3}$ & $ 0.987$ &$1.519\times 10^{-3}$ & $0.988$ \\
\hline     
\end{tabular}
\caption{  Relative $L^2$-errors and time convergence orders obtained for $(u)$  using the corrected decoupled  scheme \eqref{(6)},\eqref{T2}, the decoupled scheme \eqref{(6)},\eqref{T1}, and the decoupled scheme \eqref{T3}--\eqref{T4}.  }
\label{tab.2}
\end{center}
\end{table}
 
 In  \figref{Fig.4}, we show the evolution of the cell density by plotting the contour of the reference solution along the line $L=\lbrace0\rbrace\times(-35,35)$ at different times.

\begin{figure}[h!]
\begin{minipage}[t]{0.45\linewidth}
\includegraphics[width=\textwidth]{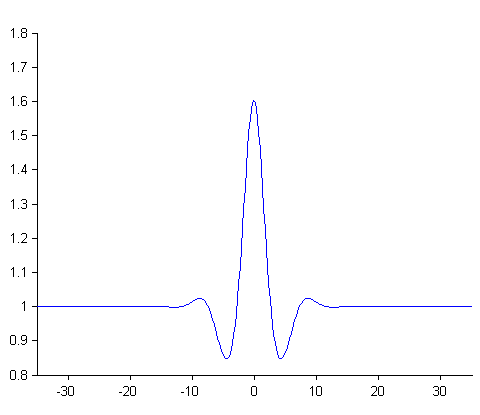}
\hspace{0.1cm}
\end{minipage}
\hspace{\fill}
\begin{minipage}[t]{0.45\linewidth}
\includegraphics[width=\textwidth]{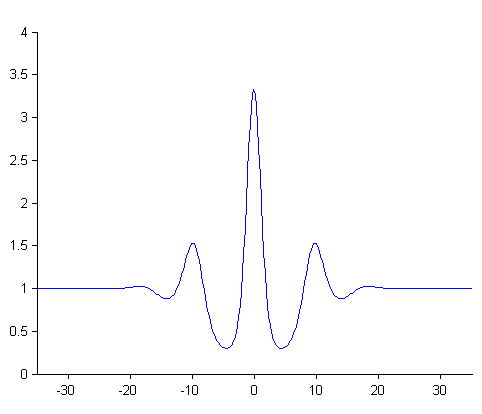}
\end{minipage}
\vspace*{0 cm}
\begin{minipage}[t]{0.45\linewidth}
\includegraphics[width=\textwidth]{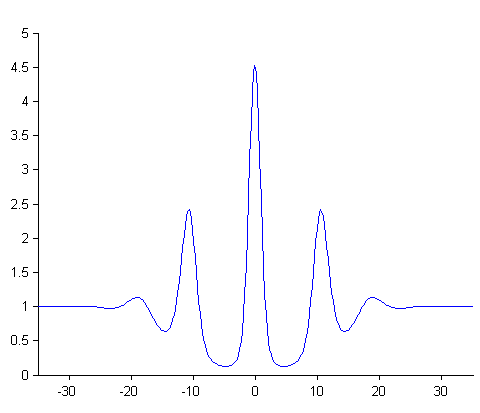}
\hspace{0.1cm}
\end{minipage}
\hspace{\fill}
\begin{minipage}[t]{0.45\linewidth}
\includegraphics[width=\textwidth]{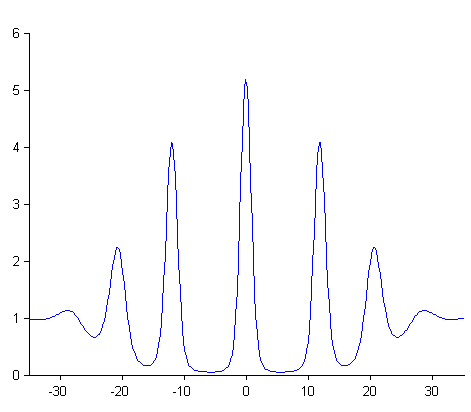}
\end{minipage}
\caption{Contours along the line $L$ of the reference solution $(u)$     at $ t=25$ (left top), $ t=75$ (right top), $ t=100$ (left bottom), $T_f$ (right bottom). }
\label{Fig.4}
\end{figure}

\textit{Test 3.} The purpose of this test is to study the accuracy of our corrected decoupled scheme when the source term in the equation for concentration $c$ is linear. To this end,  we consider the following chemotaxis-growth model  for bacterial pattern formation \cite{Aida1}:

\begin{equation}
\left\{
\begin{aligned}
 &\partial_t u = \mu \Delta u - \chi \nabla \cdot (  \, u \nabla c)+ f(u)\quad \text{in}\ \Omega \times (0,T_f) ,
 \\
&\partial_t c= \Delta c -\gamma c +u  \quad \text{in}\ \Omega \times (0,T_f),
 \end{aligned}
 \right.
 \label{S1}
\end{equation}
endowed with the homogeneous Neumann boundary conditions. We consider the following data: $\mu=0.0625$, $\chi=6$, $\gamma=16$, $f(u)=2u(1-u)$ and $T_f=30$. The computational domain is the square  $\Omega=(-8,8)^2$ discretized with a uniform mesh grid  $100 \times 100$. For all $x \in \Omega$, we choose the following initial conditions:
\begin{equation*}
  u(x,0) = \left\{\begin{array}{ll}
1+\epsilon(x) &\quad\mbox{if } \|x\|_2<0.7,\\ 
 1 &\quad\mbox{otherwise},
  \end{array}\right. \\
\end{equation*}
and  $c(x,0)=1/32$. The random perturbation $\epsilon(x)$ is defined as in \eqref{(init)}.

In this test, we compare the numerical results obtained from the decoupled scheme:
\begin{align}
& \m(K)\frac{u^{n+1}_K-u^n_K}{\Delta t}
  - \mu \sum_{\sigma\in\E_K}\tau_\sigma Du_{K,\sigma}^{n+1} \notag
  \\
  &+\chi \sum_{\substack{\sigma\in\E_{K}\\ \sigma=K|L}}\tau_\sigma\left(S\left( Dc_{K,\sigma}^{n+1}\right)u_{K}^{n+1}-S\left( -Dc_{K,\sigma}^{n+1}\right)u_{L}^{n+1}\right)-2\m(K)\,u_K^{n}\left(1-u_K^{n+1}\right)=0, \label{t1}
 \\
  &\m(K)\frac{c^{n+1}_K-c^n_K}{\Delta t} -\sum_{\sigma\in\E_K}\tau_\sigma\, Dc^{n+1}_{K,\sigma}
  +\gamma \m(K)\, c_K^{n+1}= \m(K)\, u_K^{n}, \label{t2}
\end{align}
with those of the corrected decoupled scheme, consisting of \eqref{t1} and the following equation
\begin{equation}
  \m(K)\frac{c^{n+1}_K-c^n_K}{\Delta t}-\sum_{\sigma\in\E_K}\tau_\sigma\, Dc^{n+1}_{K,\sigma}
  +\gamma \m(K)\, c_K^{n+1}= \m(K)\, u_K^{n}+ T_K^{n}.
 \label{t3} 
\end{equation} 
where for all $K\in \T$ and $n=1,...,N-1$  
 
\begin{equation*}
T_K^{n}=\m(K)\left(u_K^{n}-u_K^{n-1}\right),\quad T_K^{0}=0 . 
 \end{equation*}

The reference solution is computed by the corrected decoupled scheme using a very fine time-step size $\Delta t=10^{-4}$. The results presented in Table  \ref{tab.3} show that the corrected decoupled scheme is  highly accurate compared to the scheme \eqref{t1}--\eqref{t2}. In the case when $\Delta t=10^{-3}$, the  corrected decoupled scheme is about $32$ times more accurate than the scheme \eqref{t1}--\eqref{t2}.

\begin{table}[h!]
\begin{center}
\begin{tabular}{l c c c c} 
\hline
\\
    {$\Delta t$ } & $L^2$-error & Rate & $L^2$-error & Rate \\
                  & coorected decoupled  &       & decoupled &         \\
\hline  
\\   
     {$5.10^{-1}$} & $4.216 \times 10^{-3}$ & --- & $2.234\times 10^{-2}$ & --- \\ 
     {$10^{-3}$} & $6.022\times10^{-4}$ & $1.21$ & $1.119\times 10^{-2}$ & $ 0.430$ \\
     {$5.10^{-2}$} & $2.947\times 10^{-4}$ & $1.03$ & $6.806\times 10^{-3}$ & $ 0.717$ \\
     {$10^{-2}$} & $5.863\times 10^{-5}$ & $ 1.00$ & $1.636\times 10^{-3}$ & $ 0.886$ \\
     {$5.10^{-3}$} & $2.907\times 10^{-5}$ & $1.01$ & $8.385\times 10^{-4}$ & $0.964$ \\
     {$10^{-3}$} & $5.347\times 10^{-6}$ & $1.05$ & $ 1.707\times 10^{-4}$ & $0.989$ \\
\hline     
\end{tabular}
\caption{  Relative $L^2$-errors and time convergence orders obtained for $(u)$  using the corrected decoupled  scheme \eqref{t1},\eqref{t3} and  the decoupled scheme \eqref{t1}--\eqref{t2}.  }
\label{tab.3}
\end{center}
\end{table}

 The numerical cell density $u$ of the model computed using both schemes with $\Delta t=10^{-3}$ is shown in \figref{Fig.5}. As we can see, the solution forms  periodic  arrays of continuous rings which match well with  the patterns formed by \textit{Salmonella typhimurium} \cite{Woodward1}.

\begin{figure}[h!]
\subfigure{\includegraphics[width=6.cm]{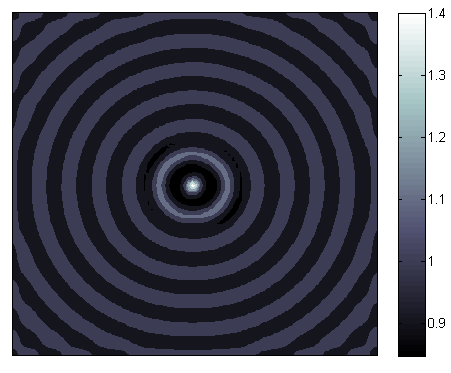}}
\subfigure{\includegraphics[width=6.cm]{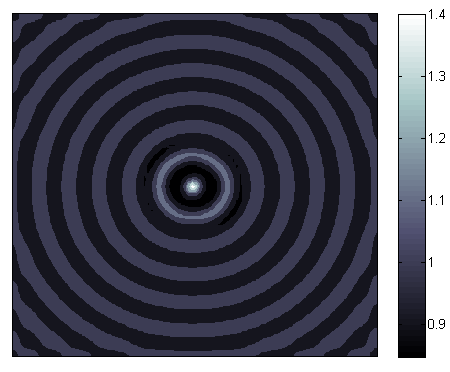}}
\caption{Cell density $(u)$ at final time $T_f$  computed via the corrected decoupled scheme \eqref{t1},\eqref{t3} (left) and the decoupled scheme \eqref{t1}--\eqref{t2} (right) with $\Delta t=10^{-3}$. }
\label{Fig.5}
\end{figure}

\textit{Test 4.} In this test, we present some numerical simulations which illustrate the ability of the presented corrected decoupled finite volume scheme to capture different forms of bacterial spatial patters. For this purpose, we consider the chemotaxis model \eqref{S1} with the  following data used in \cite{Aida1} : $\mu=0.0625$, $\gamma=32$ and $f(u)=u^2\left(1-u \right)$. The domain is the square  $\Omega=(-10,10)^2$, which is discretized via a uniform mesh grid  $150 \times 150$, and the time-step used is $\Delta t=10^{-1}$. For the final time, we take $T_f=150$ and we consider the following initial conditions 
\begin{equation*}
  u(x,0) = \left\{\begin{array}{ll}
1+\epsilon(x) &\quad\mbox{if } \|x\|_2<1,\\ 
 1 &\quad\mbox{otherwise},
  \end{array}\right. \\
\end{equation*}
and  $c(x,0)=1/32$. 

The corrected decoupled scheme used is similar to that of the previous section. However, since the logistic source $f(u)$ has now a cubic form, we use the following linearized finite volume  discretization:

\begin{align*}
 \int_K  u(x,t)^2\left(1-u(x,t) \right)\,dx \approx  
  \m(K)\,u_K^{n+1}u_K^{n}\left(1-u_K^{n}\right).
\end{align*} 

For $\chi=80$, the computed solution at $t=30$ is shown in  \figref{Fig.6}. We observe from \figref{Fig.6} (left) the formation of symmetrical spots in whole domain. The 3D view of this patterning is presented in \figref{Fig.6} (right).

\begin{figure}[h!]
\subfigure{\includegraphics[width=6.cm]{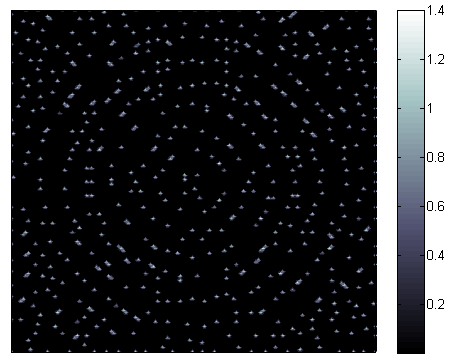}}
\subfigure{\includegraphics[width=6.cm]{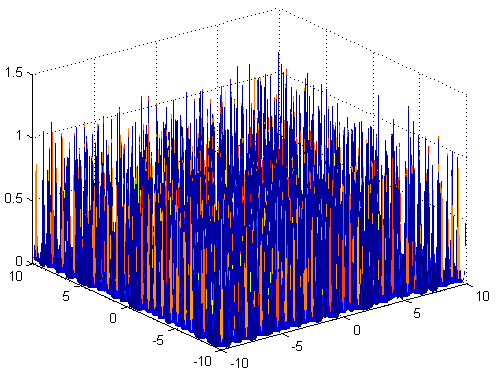}}
\caption{Cell density $(u)$ for $\chi=80$ at  $t=30$  computed via the corrected decoupled. }
\label{Fig.6}
\end{figure}

In \figref{Fig.7}, we examine the effect of the parameter $\chi$ on the numerical solution. When $\chi=6$, a honeycomb pattern is observed. Then, the solution changes its structure to continuous rings  for $\chi=7.4$. When $\chi=20$, chaotic spots appear. The symmetry of these spots increases for high values of $\chi$ (see \figref{Fig.7} (right bottom)). These symmetric spots  seem in good agreement with the \textit{Escherichia coli} patterns reported in \cite{budrene2}.

\begin{figure}[h!]
\begin{minipage}[t]{0.45\linewidth}
\includegraphics[width=\textwidth]{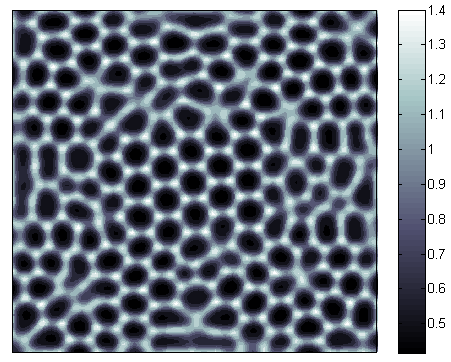}
\hspace{0.1cm}
\end{minipage}
\hspace{\fill}
\begin{minipage}[t]{0.45\linewidth}
\includegraphics[width=\textwidth]{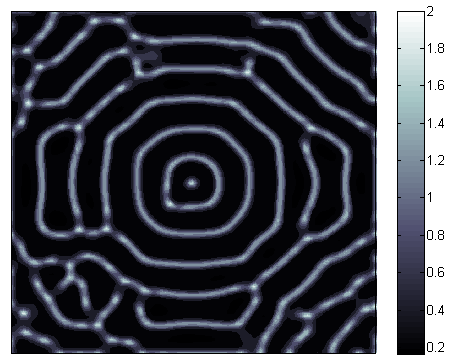}
\end{minipage}
\vspace*{0 cm}
\begin{minipage}[t]{0.45\linewidth}
\includegraphics[width=\textwidth]{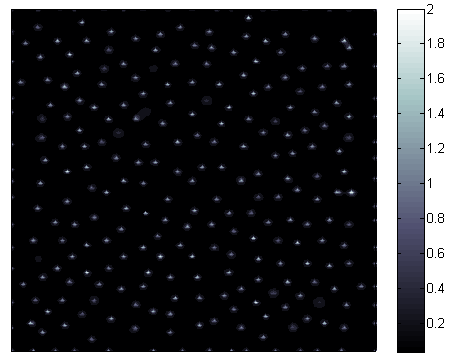}
\hspace{0.1cm}
\end{minipage}
\hspace{\fill}
\begin{minipage}[t]{0.45\linewidth}
\includegraphics[width=\textwidth]{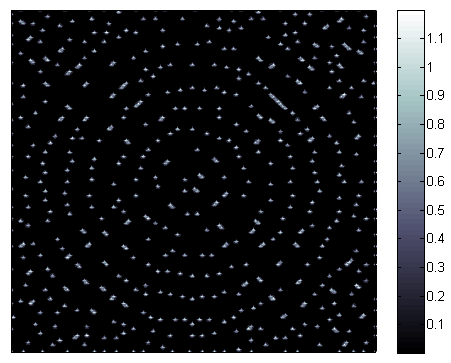}
\end{minipage}
\caption{Cell density $(u)$ computed via the corrected decoupled  scheme at final time $T_f$, for $ \chi=6$ (left top), $ \chi=7.4$ (right top), $ \chi=20$ (left bottom), $\chi=70$ (right bottom). }
\label{Fig.7}
\end{figure}

\section{Conclusion}
In this paper, a decoupled scheme for solving  chemotaxis problems is developed. Decoupled schemes are known to be very advantageous in terms of computational cost in comparison to coupled ones, however the major disadvantage of such schemes is their lack of accuracy. The proposed approximation is based on a classical decoupled scheme, which is improved by adding a suitable correction term. This approach does not affect the computational speed of the scheme and is easy to implement. Moreover, the numerical results presented show that our approach is much more accurate than usual decoupled schemes. The question is now to know how we can develop the idea of the scheme to deal with other systems of partial differential equations. This may represent an interesting topic for further research.


\section*{References}

\begin{thebibliography}{}

\bibitem{Aida1}

Aida, M., Tsujikawa, T., Efendiev, M., Yagi, A., Mimura, M.: Lower estimate of the attractor dimension for a chemotaxis growth system. J. London Math. Soc. 74(2), 453--474 (2006)

\bibitem{Akhmouch1}

Akhmouch, M., Benzakour Amine, M.: Semi-implicit finite volume
schemes for a chemotaxis-growth model. Indag. Math. 27(3), 702--720 (2016)


\bibitem{Akhmouch2}

Akhmouch, M., Benzakour Amine, M.: A time semi-exponentially fitted scheme for chemotaxis-growth models. Calcolo (2016). doi:10.1007/s10092-016-0201-4

\bibitem{Andreianov1}

Andreianov, B., Bendahmane, M., Saad, M.: Finite volume methods for degenerate chemotaxis model. J. Comput. Appl. Math. 235 (14), 4015--4031 (2011)

\bibitem{budrene2}
Budrene, E.O., Berg, H.C.: Dynamics of formation of symmetrical patterns by chemotactic bacteria. Nature 376, 49--53 (1995)

\bibitem{Hillairet1}

Chainais-Hillairet, C., Liu, J.-G., Peng, Y.-J: Finite volume scheme for multi-dimensional drift-diffusion equations and convergence analysis. Math. Mod. Numer. Anal. 37, 319--338 (2003)  


\bibitem{Oster1}

Oster, G.F., Murray, J.D.: Pattern Formation Models and Developmental Constraints. J. expl. Zool. 251, 186-202 (1989)

\bibitem{Keller1}

Keller, E.F., Segel, L.A.: Travelling bands of chemotactic bacteria: a theoretical analysis. J. Theor. Biol. 30, 235-248 (1971)

\bibitem{Murray1}

Murray, J.D., Deeming, D.C., Ferguson, M.W.J.: Size dependent pigmentation pattern formation in  embryos of \textit{Alligator mississippiensis}: time  of initiation  of pattern generation mechanism. Proc. R. Soc. B 239, 279-293 (1990)

\bibitem{Myerscough1}

Myerscough, M.R., Murray, J.D.: Analysis of propagating pattern in a chemotaxis system. Bull. math. Biol. 54, 77-94 (1992)





\bibitem{Patankar1}

Patankar, S.V.: Numerical Heat Transfer and Fluid Flow, Hemisphere Publishing Corporation, Taylor and Francis Group, New York (1990)


\bibitem{Saito1}

Saito, N.: Conservative upwind finite-element method for a simplified {Keller-Segel} system modelling chemotaxis, IMA J. Numer. Anal. 27, 332-365 (2007)

\bibitem{Saito2}

Saito, N., Suzuki, T.: Notes on finite difference schemes to a parabolic-elliptic system modelling chemotaxis. Appl. Math. Comput. 171(1), 72--90 (2005) 

\bibitem{Strehl1}

Strehl, R., Sokolov, A., Kuzmin, D., Turek, S.: A flux-corrected finite element method for chemotaxis problems. Comput. Methods Appl. Math. 10(2), 219--232 (2010) 



\bibitem{Filbet1}

Filbet, F.: A finite volume scheme for the {Patlak-Keller-Segel} chemotaxis model, Numer. Math. 104(4), 457-488 (2006)




\bibitem{Chamoun1}

Chamoun, G., Saad, M., Talhouk, R.: Monotone combined edge finite volume-finite element scheme for anisotropic Keller-Segel model. Numer. Methods Partial Differential Equations 30 (3), 1030-1065 (2014)

\bibitem{Ibrahim1}

Ibrahim, M., Saad, M.: On the efficacy of a control volume finite element method for
the capture of patterns for a volume-filling chemotaxis model. Comp. Math. Appl. 68, 1032–1051 (2014)

\bibitem{Eymard1}

Eymard, R., Gallou\"et, T., Herbin, R.: Finite volume methods. In: P. G. Ciarlet and J. L. Lions(eds.), Handbook of numerical analysis volume VII, 713-1020, North-Holland (2000)









\bibitem{chatard2}

Bessemoulin-Chatard, M., Chainais-Hillairet, C., Filbet, F.: On discrete functional inequalities for some finite volume schemes. IMA J. Numer. Anal. 35(3), 1125--1149 (2015) 

\bibitem{Spalding1}

Spalding, D.B.: A novel finite difference formulation for differential expressions involving both first and second derivatives. Int. J. Numer. Methods Eng. 4, 551--559 (1972)

\bibitem{Woodward1}

Woodward D., Tyson R., Myerscough M., Murray J., Budrene E., Berg H.: Spatio-temporal patterns generated by S. typhimurium. Biophys. J. 68, 2181--2189 (1995)


\end{thebibliography}
\end{document}